\documentclass[10pt]{amsart}
\usepackage{mathtools}
\usepackage[mathscr]{eucal}
\usepackage{graphics}
\usepackage[all]{xy}
\usepackage{caption, subcaption}
\usepackage[usenames,dvipsnames]{color}
\usepackage{graphicx}
\usepackage[text={5.58in,8.5in},centering,letterpaper,dvips]{geometry}
\usepackage{amsfonts}
\usepackage{epsf}
\usepackage{amssymb}
\usepackage{amsmath}
\usepackage{amscd}
\usepackage{pdfpages}
\usepackage{fancyhdr}
\usepackage{setspace}
\usepackage[all]{xy}
\usepackage{verbatim}
\usepackage{enumerate}
\usepackage{pgf}
\usepackage[latin1]{inputenc}
\usepackage[colorlinks=true, urlcolor=NavyBlue, linkcolor=NavyBlue, citecolor=NavyBlue]{hyperref}

\theoremstyle{theorem}
\newtheorem{theorem}{Theorem}[section]

\newtheorem{lemma}[theorem]{Lemma}

\newtheorem{question}[theorem]{Question}
\newtheorem{problem}[theorem]{Problem}

\newtheorem{thmx}{Theorem}
\newtheorem{corox}[thmx]{Corollary}

\makeatletter
\newtheorem*{rep@theorem}{\rep@title}
\newcommand{\newreptheorem}[2]{%
\newenvironment{rep#1}[1]{%
 \def\rep@title{#2 \ref{##1}}%
 \begin{rep@theorem}}%
 {\end{rep@theorem}}}
\makeatother

\newreptheorem{theorem}{Theorem}
\newreptheorem{lemma}{Lemma}
\newreptheorem{question}{Question}
\newreptheorem{corollary}{Corollary}
\newreptheorem{proposition}{Proposition}

\theoremstyle{definition}
\newtheorem{definition}[theorem]{Definition}
\newtheorem{example}[theorem]{Example}
\newtheorem{remark}[theorem]{Remark}

\newcommand{\Z}{\mathbb{Z}}

\newcommand{\T}{\mathbb T}
\newcommand{\RP}{\mathbb{RP}}
\newcommand{\CP}{\mathbb{CP}}

\newcommand{\Dd}{\mathcal D}
\newcommand{\Tt}{\mathcal T}

\newcommand{\Mm}{\mathcal M}

\newcommand{\Ss}{\mathcal S}

\newcommand{\wh}{\widehat}

\begin{document}

\rhead{\thepage}
\lhead{\author}
\thispagestyle{empty}


\raggedbottom
\pagenumbering{arabic}
\setcounter{section}{0}


\title[Three-fold branched covers of $S^4$]{Note on three-fold branched covers of $S^4$}

\author[R.\ Blair]{Ryan Blair}
\address{Department of Mathematics and Statistics, California State University, Long Beach, USA}
\email{ryan.blair@csulb.edu}

\author[P.\ Cahn]{Patricia Cahn}
\address{Department of Mathematics and Statistics, Smith College, USA}
\email{pcahn@smith.edu}

\author[A.\ Kjuchukova]{Alexandra Kjuchukova}
\address{Max Planck Institute for Mathematics, Germany}
\email{sashka@mpim-bonn.mpg.de}

\author[J.\ Meier]{Jeffrey Meier}
\address{Department of Mathematics, Western Washington University, USA}
\email{jeffrey.meier@wwu.edu}



\begin{abstract}
	We show that any 4--manifold admitting a $(g;k_1,k_2,0)$--trisection is an irregular 3--fold cover of the 4--sphere whose branching set is a surface in $S^4$, smoothly embedded except for one singular point which is the cone on a link.
	A 4--manifold admits such a trisection if and only if it has a handle decomposition with no 1--handles; it is conjectured that all simply-connected 4--manifolds have this property.	
\end{abstract}

\maketitle

\section{Introduction}
\label{sec:Intro}

A classical result in 3--manifold topology states that every closed, oriented 3--manifold can be realized as an irregular 3--fold cover of the 3--sphere branched along a knot.  This was proved independently by Hilden~\cite{hilden1974every}, Hirsch~\cite{hirsch1974offene}, and Montesinos~\cite{montesinos1974representation} in 1974. 
Irregular 3--fold covers with smooth branching sets are characterized by the property that every point on the branch locus has two pre-images, one of branching index 1 and one of branching index 2. Thus, they are special cases of simple branched covers~\cite{piergallini1995four} and of irregular dihedral branched covers~\cite{CS1984linking, kjuchukova2018dihedral}. 

Hilden's proof leveraged the way in which Heegaard splittings of 3--manifolds arise naturally as branched covers of bridge splittings of links.  Trisections, which were introduced by Gay and Kirby as a 4--dimensional analog of Heegaard splittings~\cite{gaykirby2016trisections}, 
have a similar relationship with bridge trisections of knotted surfaces~\cite{mz-bridge1,mz-bridge2}; see also~\cite{cahnkjuchukova2017singbranchedcovers,LamMei_18_Bridge-trisections,LamMeiSta_20_Symplectic-4-manifolds-admit}. We extend Hilden's techniques to dimension four using trisections.  We make use of the fact that trisections are well-behaved under branched coverings:  if $\pi\colon X\to X'$ is a (singular) branched covering and $X$ and $X'$ have trisections $\T$ and $\T'$, respectively, then we say that \emph{ $\T$ is a (singular) branched cover of $\T'$} if $\pi$ respects the decompositions of $X$ and $X'$ induced by the respective trisections, restricts to a branched covering from each piece of $\T$ to the corresponding piece of $\T'$, and the (singular) branch locus of $\pi$ is in bridge position with respect to $\T'$.

\begin{thmx}
\label{thmx:main}
	Let $X$ be a 4--manifold that admits a $(g;k_1,k_2,0)$--trisection $\T$.  Then $\T$ is an irregular 3--fold cover of the standard trisection $\mathbb{T}_0$ of $S^4$. The branching set $\Ss$ is a closed, connected surface, which is smoothly embedded in $S^4$ away from one singular point, the cone on a link.
	
	If, in addition, $k_2=0$, the branching set $\Ss'$ can be chosen to be embedded in $S^4$. In this case, there are two singular points, each the cone on a knot.
\end{thmx}

There is a well-known correspondence between trisections and handle decompositions of smooth 4--manifolds (see, for example,~\cite[Section~4]{meier2016classification}), and this gives rise to the following corollary.

\begin{corox}
\label{corox:main}
	Let $X$ be a closed smooth 4--manifold built with no 1--handles. Then $X$ is an irregular 3--fold cover of $S^4$.
	The branching set $\Ss$ is a closed, connected surface which is smoothly embedded in $S^4$ away from one singular point, the cone on a link.
	
	If $X$ is a 4-manifold built with no 1--handles and no 3--handles, the branching set $\Ss'$ can be chosen embedded in $S^4$. In this case, there are two singular points, each the cone on a knot.
\end{corox}

\begin{remark}
\label{rmk:euler}
	The Euler characteristic of the branching sets in Theorem~\ref{thmx:main} and Corollary~\ref{corox:main} is easily computed. In the case of one singularity, we have: $$\chi(\Ss)= k_1+k_2-g+3 =5-\chi(X).$$ 
	
	In the case where the branching set $\Ss'$ is embedded with two singularities, we have: $$\chi(\Ss')=k_1-g+2 = 4-\chi(X).$$
\end{remark}

\begin{proof}[Proof of Corollary~\ref{corox:main}]
	If $X$ admits a $(g;k_1,k_2,k_3)$--trisection, then $X$ admits a handle decomposition with  a single 0--handle, $k_3$ 1--handles,  $g-k_1$ 2--handles, $k_2$ 3--handles, and a single 4--handle~\cite[Section~4]{meier2016classification}. The statements of the corollary follow by applying Theorem~\ref{thmx:main} to this trisection in the cases where $k_3=0$ or $k_3=k_2=0$.
\end{proof}

In 1978, Montesinos showed that 4--manifolds with boundary which are built with 0--, 1--, and 2--handles arise as irregular 3--fold covers of $B^4$, branched along ribbon surfaces~\cite{montesinos1978}. 
The bridge trisected surface $\Ss$ produced in Theorem~\ref{thmx:main} becomes a ribbon surface in $B^4$ once an open neighborhood of the singular point is removed.  Thus, Theorem~\ref{thmx:main} can be viewed as a trisection-theoretic counterpart for closed 4--manifolds to the handle-theoretic Theorem~6 of~\cite{montesinos1978}.
A weaker version of the first statement of our theorem can be proved using Montesinos's techniques, showing that a {\it stabilization} of the given trisection $\T$ of $X$ covers $\T_0$.
The first statement of Corollary~\ref{corox:main} can be obtained using Theorem~6 of~\cite{montesinos1978}.

Families of 3--fold covers of $S^4$ by simply-connected manifolds where the branch sets are embedded with cone singularities are given in~\cite{cahnkjuchukova2017singbranchedcovers, cahnkju2018genus, blair2018simply}.
In these constructions, there is at most one singularity on each cover. Moreover, this singularity can be arranged to be the cone on a knot so the branching set is embedded.
It is an open question whether Theorem~\ref{thmx:main} can be strengthened to ensure that the singularity on the branching set is the cone on a knot or so that the singularity is removed altogether, yielding a branching set that is smooth.

Section~\ref{examples} of this paper contains some new examples of (singular) irregular 3--fold coverings of $S^4$.  While many standard manifolds can be represented by combining the examples in this section, one notable exception is K3.  In~\cite{LamMei_18_Bridge-trisections}, trisections for algebraic surfaces were introduced that satisfy the hypotheses Theorem~\ref{thmx:main} and can, therefore, be described as irregular 3--fold singular branched coverings.  It is natural to ask what the branch loci are realizing these trisections as covers of $\T_0$.

If the 11/8--Conjecture is true, then every simply-connected, closed, oriented, smooth  4--manifold is homeomorphic to an irregular 3--fold branched cover of $S^4$, as we show next.
Consider the collection $\Mm$ of 4--manifolds that contains $\CP^2$, $S^2\times S^2$, and K3 and is closed under the operations of taking connected sums and reversing orientations. Each member of $\Mm$ admits a $(g,0)$--trisection~\cite{LamMei_18_Bridge-trisections,SprTil}. Every smooth, simply-connected 4--manifold that satisfies the 11/8--Conjecture is homeomorphic to a member of $\Mm$~\cite{freedman1982topology,FreQui90}; for details, see~\cite[Section~1.2]{GomSti}. Thus, we obtain to following result.

\begin{corox}
\label{corox:types}
	Every simply-connected, closed, oriented, smooth  4--manifold $X$ with $b_2(X)\geq \frac{11}{8}|\sigma(X)|$ is homeomorphic to an irregular 3--fold branched cover of the 4--sphere.
\end{corox}

On the other hand, there is no hope to realize {\it all} homeomorphism types of smooth 4--manifolds as 3--fold covers over the 4--sphere, even if arbitrarily complicated branching sets are allowed. A lower bound on the degree of a branched cover $f\colon X\to S^4$ is obtained from the length of the reduced cohomology ring of $X$ with rational coefficients~\cite[Theorem~2.5]{berstein1978degree}. It follows that a branched cover of $S^4$ by the 4--torus $T^4$ has degree at least four. In turns out that every closed, oriented, smooth 4--manifold is a simple 4--fold cover of the 4--sphere~\cite{piergallini1995four}.
It would be interesting to have a version of Theorem~\ref{thmx:main} for simple 4--fold covers in which the branch locus is found to be immersed and non-singular, in agreement with~\cite{iori2002}.  Branched covers of trisections in which the branch locus is immersed and singular are studied in~\cite{LamMeiSta_20_Symplectic-4-manifolds-admit}, where they are used to construct trisections for symplectic 4--manifolds that are compatible with the ambient symplectic structure.

Exotic pairs of 4--manifolds admitting $(g,0)$--trisections are constructed in~\cite{LamMei_18_Bridge-trisections}. By our Theorem~\ref{thmx:main}, these manifolds are irregular 3--fold branched covers of the 4--sphere. Moreover, a 4--manifold admits a $(g;k_1,k_2,0)$--trisection if and only it can be built without 1--handles~\cite{meier2016classification}. It is an open question whether every closed simply-connected 4--manifold admits a $(g,0)$--trisection or a handle decomposition of this type~\cite{LamMei_18_Bridge-trisections}.  This motivates the following question.

\begin{question}
\label{ques:3-fold}
	Is every simply-connected, smooth 4--manifold an irregular 3--fold cover of $S^4$ with branch set a surface that is smooth and embedded away from one singular point?
\end{question}

In light of our main result, a negative answer to Question~\ref{ques:3-fold} would amount to discovering a simply-connected 4--manifold that has no handle decomposition without 1--handles. Thus, the above is a formulation of Problem 4.18 on the Kirby List~\cite{Kir_78_Problems-in-low-dimensional} in terms of the classification of 3--fold irregular branched covers of $S^4$.

\subsection*{Acknowledgements.}
This work was started at the AIM workshop ``Symplectic four-manifolds through branched coverings". It was partly completed while A.K. and J.M. were visiting the MPIM in Bonn. We are grateful to American Institute of Mathematics and the Max Planck Society for their support. The authors are partially supported by NSF grants DMS-1821254, DMS-1821212, DMS-1821257 and DMS-1933019, respectively.

\section{Background}
\label{sec:back}

In this section, we review the requisite background material relating to trisections, singular surfaces, and branched covers. 

\subsection{Trisections}
\label{subsec:tris}

The theory of trisections was introduced by Gay and Kirby~\cite{gaykirby2016trisections} who showed that every smooth, orientable, closed, connected 4--manifold admits a trisection, defined below. We also refer the reader to~\cite{Gay-notes,meier2016classification} for more thorough introductions.

A \emph{$(g;k_1,k_2,k_3)$--trisection} of $X$ is a decomposition $X=Z_1\cup Z_2\cup Z_3$ such that
\begin{enumerate}
	\item $Z_i\cong \natural^{k_i}(S^1\times B^3)$,
	\item $H_i = Z_i\cap Z_{i+1} \cong \natural^g(S^1\times D^2)$, and
	\item $\Sigma = Z_1\cap Z_2\cap Z_3 \cong \#^g(S^1\times S^1)$.
\end{enumerate}
The parameter $g$ is the \emph{genus} of the trisection, and the surface $\Sigma$ is called the \emph{core}.
The subcomplex $H_1\cup H_2\cup H_3$ is called the \emph{spine}, and the 4--dimensional pieces $Z_i$ are the \emph{sectors}. An important feature of the theory is that the spine of a trisection determines the ambient 4--manifold.
In the case that the $k_i=k$ for each $i$, we call the trisection \emph{balanced} and refer to it as a $(g,k)$--trisection; otherwise, the trisection is called \emph{unbalanced}.  

The notion of an unbalanced trisection and the connection with handlebody decompositions was explored in~\cite{meier2016classification}.
It is shown that a 4--manifold that admits a $(g;k_1,k_2,k_3)$--trisection can be built with one 0--handle, $k_1$ 1--handles, $g-k_2$ 2--handles, $k_3$ 3--handles, and one 4--handle.
Thus, if $k_i=0$ for some $i=1$, 2, or 3, then $X$ is simply-connected.

\begin{remark}
	 The trisection parameters $(g;k_1,k_2,k_3)$ are additive under connected sum of trisected 4--manifolds~\cite{gaykirby2016trisections}. The 4--manifolds $S^4$, $\pm\CP^2$, $S^2\times S^2$ and K3 admit $(g,0)$--trisections with $g=0$, 1, 2 and 22, respectively~\cite{gaykirby2016trisections, LamMei_18_Bridge-trisections,SprTil}. 	
	 As before, we denote the  $g=0$ trisection of $S^4$ by~$\T_0$. 
\end{remark}

\subsection{Singular bridge trisections and branched covers of $S^4$}
\label{subsec:sing-bt}

We define the branched covers constructed in the proof of Theorem~\ref{thmx:main}.

A surface $\Ss\subset S^4$ is \emph{singular} provided that
\begin{enumerate}
	\item $\Ss$ is smoothly embedded in $S^4$ away from finitely many points; 
	\item in a neighborhood of each non-smooth point, $\Ss$ is given as the cone on a smooth link in $S^3$.
\end{enumerate}
The surface $\Ss$ is embedded in $S^4$ when its singularities are cones on knots.

The notion of a bridge trisection was introduced in~\cite{mz-bridge1} and extended in~\cite{mz-bridge2}.
We adopt the following definition of a bridge trisection for a singular surface in $S^4$.

\begin{definition}
	Let $\Ss\subset S^4$ be a singular surface.  We say that $\Ss$ is in \emph{$b$--bridge trisected position} with respect to the genus zero trisection $\T_0$ if for each $i\in\Z_3$
	\begin{enumerate}
		\item $\Tt_i = H_i\cap\Ss$ is a trivial, $b$--strand tangle;
		\item $\Dd_i = Z_i\cap \Ss$ is either the cone on a  link or a trivial disk system, that is, a collection of properly embedded, smooth 2--disks in $Z_i\cong B^4$ which are simultaneously boundary parallel.
	\end{enumerate}
	The decomposition
	$$(S^4,\Ss) = (Z_1,\Dd_1)\cup(Z_2,\Dd_2)\cup(Z_3,\Dd_3)$$
	is called a \emph{singular $b$--bridge trisection}.
\end{definition}

Note that, as defined, a singular surface that is in bridge trisected position has at most three singularities. This suffices for our branched cover construction. It is possible to generalize the definition, allowing transversely immersed surfaces with any number of components, and with any number of singularities, to be put in bridge trisected position. 
We refer the reader to~\cite{cahnkjuchukova2017singbranchedcovers,LamMei_18_Bridge-trisections, cahnkju2018genus, LamMeiSta_20_Symplectic-4-manifolds-admit} for related discussions about singular branched coverings of bridge trisected surfaces.

\begin{definition}
\label{def:sing-br}
	Let $\Ss \subset S^4$ be a singular surface.  A map $f\colon X\to S^4$ is called a \emph{singular branched cover} of $S^4$ with \emph{branch set} $\Ss$ if the following two conditions hold:
	\begin{enumerate}
		\item Away from the union of small 4--ball neighborhoods around each of the singular points, $f$ is a branched covering map.
		\item In a 4--ball neighborhood of a singular point, $f$ is the cone on a branched covering map from a 3--manifold $M$ to $S^3$.
	\end{enumerate}
\end{definition}

In the above definition, for each neighborhood of a singular point $z\in \Ss$, the manifold $M$ is the cover of $S^3$ branched along the link describing the singularity $z$. When the total space $X$ is a manifold, that is, in all of our constructions, $M$ is necessarily homeomorphic to $S^3$. The general case, where $X$ is a stratified space with isolated singularities, is studied in~\cite{geske2018signatures}. We also refer the reader to~\cite{berstein1978degree,Zud} for broad overviews of branched coverings. 

\begin{definition}
\label{def:tri_branch}
	Given a branched covering $\pi\colon X\to X'$ and trisections $\T$ and $\T'$ for $X$ and $X'$, respectively, we say that \emph{ $\T$ is a (singular) branched cover of $\T'$} if, for each $i\in\Z_3$
	\begin{enumerate}
		\item $\pi(Z_i)=Z_i'$;
		\item $\pi\vert_{Z_i}\colon Z_i\to Z_i'$ is a (singular) branched cover; and
		\item the (singular) branch locus of $\pi$ is in bridge position with respect to $\T'$.
	\end{enumerate} 
\end{definition}

Note that criterion (3) is essential when dealing with irregular branched coverings, since, for example, there are 2--strand tangles that are not trivial, but that have $B^3$ as an irregular 3--fold cover.
Even if the branched cover is cyclic, criterion (3) is still required, since there are disks in $B^4$ bounded by the unknot whose 2--fold cover is $B^4$~\cite{Gor_74_On-the-higher-dimensional-Smith-conjecture}.

\subsection{Fox colorings and Hilden's map}
\label{subsec:Fox-Hilden}

A branched cover is determined by its ordinary covering behavior away from the branching locus~\cite{fox1957covering}, which in turn arises from a homomorphism with domain the fundamental group of the branch set complement. When considering (connected) 3--fold irregular covers of $S^3$ with branching set a link $L$, the homomorphism in question maps onto $D_3$, the dihedral group of order 6. Meridians of $L$ are sent to reflections. As is well known, a homomorphism $\rho\colon \pi_1(S^3\backslash L)\twoheadrightarrow D_3$ can be represented by a Fox 3--coloring of $L$. We denote the three ``colors" by 1, 2 and 3, each identified with a reflection in $D_3$. A color can be assigned to every meridian of $L$, according to its image under $\rho$. In a diagram, each arc is colored by the image of its Wirtinger meridian. A Fox 3--coloring of a singular tri-plane diagram is defined in an analogous way, and it determines an irregular 3--fold cover of $S^4$ branched along a singular surface~\cite{cahnkjuchukova2017singbranchedcovers}. 

\begin{lemma}\label{lemma:unlink-lift}
	Suppose that $L$ is a $(k+2)$--component unlink with one component colored ``1'' and the rest colored ``2''.  Let $f\colon Y\to S^3$ denote the irregular 3--fold covering corresponding to this coloring.  Then $Y\cong\#^k(S^1\times S^2)$.  Moreover, for $g\geq k$ there exists a $(g+2)$--bridge position for $L$ such that the bridge sphere lifts to a genus $g$ Heegaard surface for $Y$.
\end{lemma}

\begin{proof}
	Given $L$ as above, denote by $f\colon Y\to S^3$ the irregular 3--fold branched cover induced by the specified coloring on the branching set $L$. We consider two Heegaard splittings for $Y$, corresponding to two bridge positions for $L$. 

	First, let $\Tt_1$ be a $(k+2)$--strand trivial tangle represented by a diagram with no crossings. Let one arc of $\Tt_1$ be colored ``1'' and the remaining $k+1$ arcs colored ``2''. Doubling  $\Tt_1$ along its boundary gives a $(k+2)$--bridge decomposition of $L = \Tt_1\cup \overline{\Tt_1}$, with bridge sphere $S^2_1$. The link $L$ meets $S^2_1$ in $2k+4$ points which are branch points for $f$. Each of these branch points has two pre-images under the 3--fold cover $f$, so it follows that the Euler characteristic of $\Sigma_1 = f^{-1}(S^2_1)$ is $2-2k$. Since the branched cover of a trivial tangle is a handlebody, we have that $\Sigma_1$ is a genus $k$ Heegaard surface for $Y$. Moreover, this Heegaard splitting for $Y$ is obtained by doubling a genus $k$ handlebody along its boundary. Therefore, $Y\cong \#^k(S^1\times S^2)$.

	Next, let $g\geq k$ and endow the link $L$ with the same coloring as before, inducing the same branched covering space $Y$. Furthermore, assume $L$ is in $(g+2)$--bridge position with respect to a bridge sphere $S^2_2$.
	As before, we can compute the Euler characteristic of $\Sigma_2 = f^{-1}(S^2_2)$ and conclude that $\Sigma_2$ is a genus $g$ Heegaard surface for $Y$.
\end{proof}

For concreteness, we illustrate the standard $(g+2)$--bridge position for the $k$--component unlink $L$, together with the 3--coloring described above, in Figure~\ref{fig:unlinkbridge}.

\begin{figure}[h!]
	\includegraphics[scale=.5]{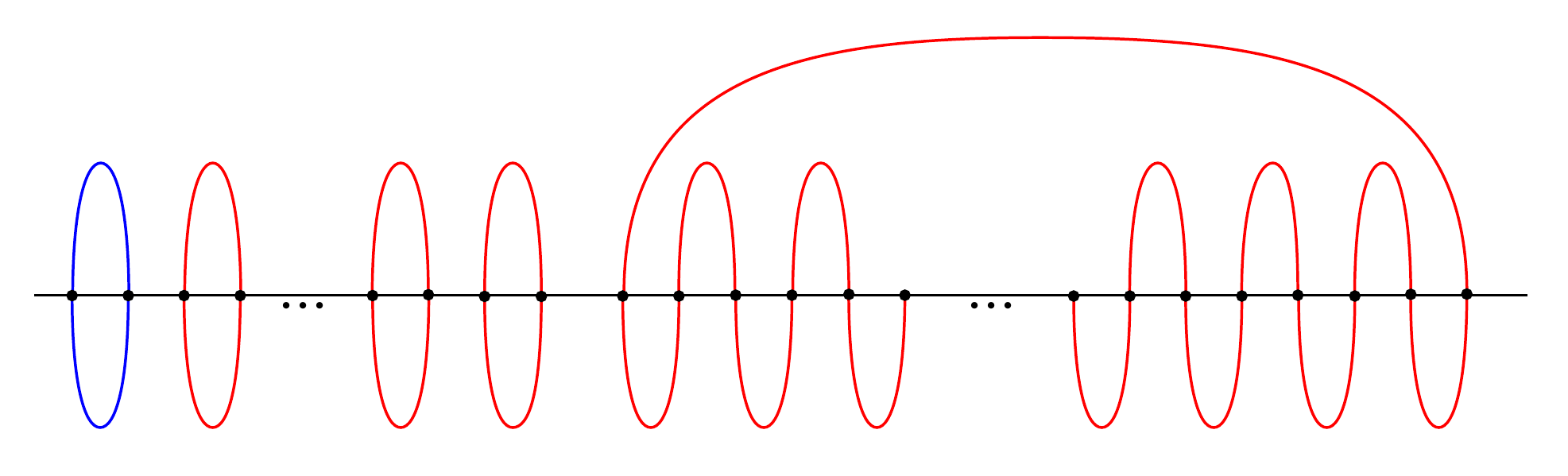}
	\caption{A 3--coloring of $(k+2)$-component unlink $L$ in $(g+2)$--bridge position, with $g>k$.} 
	\label{fig:unlinkbridge} 
\end{figure}

\section{Proof of the main theorem}
\label{sec:proof}

In this section, we summarize several results in Hilden's work~\cite{hilden1974every} and prove our main theorem.  Given an oriented manifold $X$ or a pair of oriented manifolds $(X,Y)$, we let $\overline X$ and $\overline{(X,Y)}$ denote the same objects, but with the orientation reversed.

\subsection{Revisiting Hilden's result}
\label{subsec:Hilden}

We restate the main results of Hilden's paper~\cite{hilden1974every} in language that is well-suited for carrying out our generalization to dimension four. The statement below follows from Theorems~8, ~9 and ~10 of~\cite{hilden1974every}, and their proofs.  We denote by  $\pi \colon \Sigma_g \to S^2$ the standard covering map described by Hilden.
This map realizes the genus $g$ surface as an irregular 3--fold covering of $S^2$ with a branching set consisting of $2g+4$ points, whose union we denote by~$\bold x$. The map $\pi$ is constructed starting with a genus $3g+1$ handlebody $H_{3g+1}$ and a 6--fold cover $H_{3g+1}\to B^3$ with an action by the dihedral group of order six. A $\Z/2\Z$ quotient of this map defines the irregular 3--fold cover $H_g\to B^3$ whose boundary is the map $\pi$; see Figure~\ref{hildenmap}.

\begin{figure}[htbp] 
	\includegraphics[width=4in]{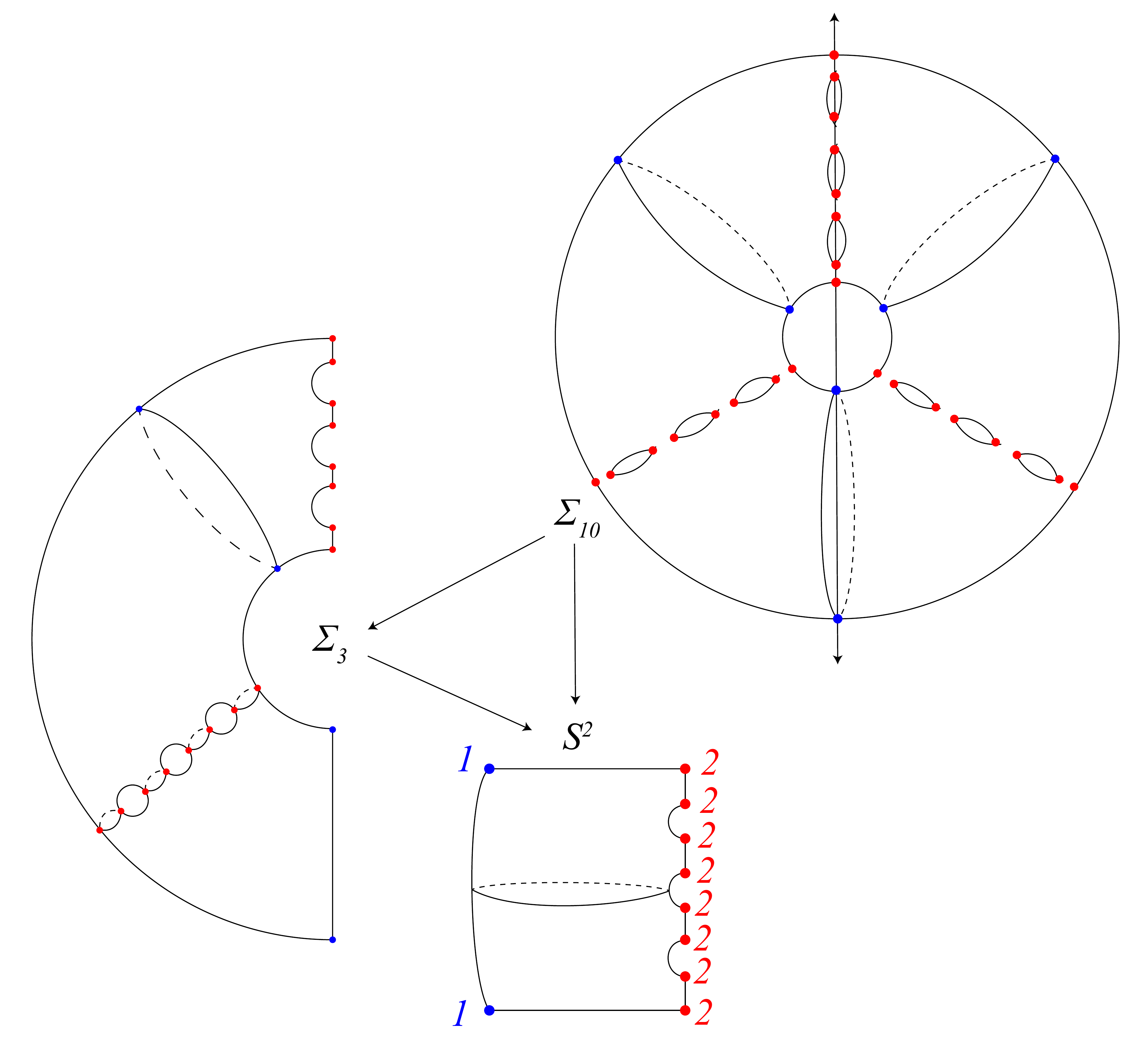}
	\caption{Hilden's 3--fold irregular cover of $S^2$ by a surface $\Sigma_3$ of genus 3.  $\Sigma_3$ is a $\mathbb{Z}/2\mathbb{Z}$ quotient of a 6--fold regular dihedral cover $\Sigma_{10}$ of $S^2$.}
	\label{hildenmap}
\end{figure}

\begin{theorem}[Hilden~\cite{hilden1974every}]
\label{thm:Hilden}
	Let $H$ be any genus $g$ handlebody with $\partial H = \Sigma_g$. Denote Hilden's covering map by $\pi\colon \Sigma_g \to S^2$.
	Given a pairing $P$ of the branch points $\bold x$ in $S^2$, there exists an extension $\widehat{\pi}\colon H\rightarrow B^3$ of $\pi$ whose branch set is a trivial tangle $\Tt$ inducing the pairing~$P$.
\end{theorem}

\subsection{Normalized bridge splittings and Piergallini moves}
\label{subsec:pierg}

Let $L_1$ and $L_2$ be two 3--colored link diagrams representing the same 3--manifold as a branched cover of $S^3$. Piergallini gave a set of four moves that suffice to convert $L_1$ to $L_2$, up to colored isotopy~\cite[Figures~1 and~7]{pierg91}.
The moves are applied to {\it normalized diagrams}. A colored link diagram is {\it normalized} if it is the plat closure of a braid, with maxima and minima as shown in Figure ~\ref{fig:normalized}; the leftmost maximum and minimum are colored ``1'', and all remaining extrema are colored ``2''. We say a link, together with a choice of bridge splitting and representation $\pi_1(S^3-L)\twoheadrightarrow D_3$, is {\it normalized} if it admits a normalized diagram.  While Piergallini's moves are defined on diagrams \cite{pierg91}, they can as well be applied to links in normalized bridge position in the obvious way. 

\begin{figure}[htbp]
	\includegraphics[width=2in]{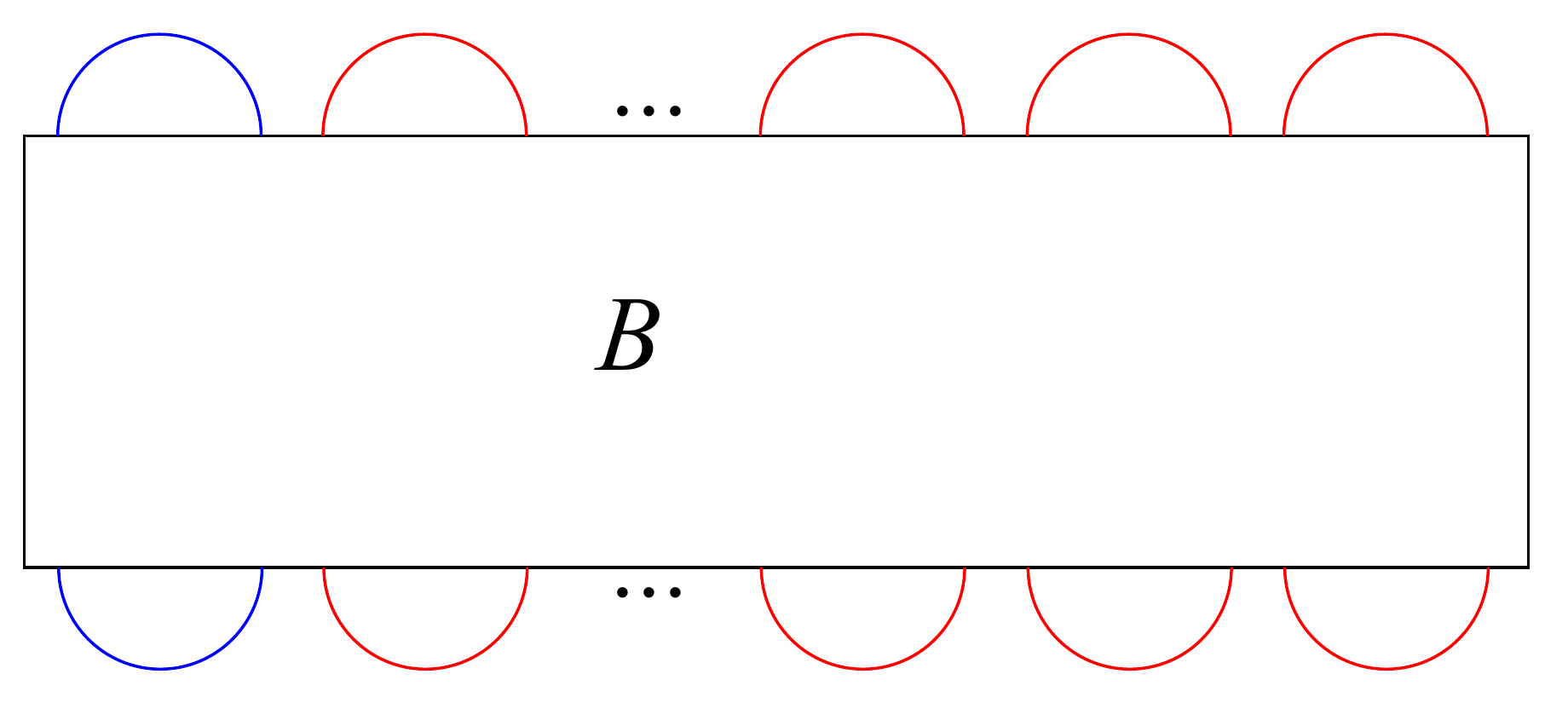}

	\caption{A normalized link diagram.  $B$ is a 3--colored braid.}	\label{fig:normalized}
\end{figure}

The following theorem follows from Piergallini's proof in ~\cite{pierg91}. 

\begin{theorem}
\label{thm:Pier_splitting}
	Suppose that $H_1\cup_\Sigma\overline{H_2}$ and $H_1'\cup_{\Sigma'}\overline{H_2'}$ are two homeomorphic Heegaard splittings of (homeomorphic) 3--manifolds $M$ and $M'$ that are respectively given as irregular 3--fold covers of normalized bridge splittings $(S^3,L) = (B_1,\Tt_1)\cup_{(S,x)}\overline{(B_2,\Tt_2)}$ and $(S^3,L') = (B_1',\Tt_1')\cup_{(S',\bold x')}\overline{(B_2',\Tt_2')}$ of links $L$ and $L'$ in $S^3$.  Then $(B_2,\Tt_2)$ can be transformed via Piergallini moves and isotopy rel-$\partial$ to a tangle $(B_2,\Tt_2'')$ such that $L'' = \Tt_1\cup_\bold x\overline{\Tt_2''}$ is isotopic to $L'$.
\end{theorem}

\begin{proof}
	First, we isotope the normalized bridge splitting $(B_1',\Tt_1')\cup_{(S',\bold x')}\overline{(B_2',\Tt_2')}$ to a normalized bridge splitting $(B_1,\Tt_1)\cup_{(S,\bold x)}\overline{(B_2,\Tt_2'')}$.  In other words, we arrange that the given two normalized bridge splittings to agree as normalized tangles in $B_1$.  Having done this, we are in position to apply the techniques of Section~3 of~\cite{pierg91}, where the claim that $\Tt_2$ and $\Tt_2''$ are related by Piergallini moves and isotopy rel-$\partial$ is proved, though this phrasing is not explicit.
\end{proof}

\subsection{Proof of Theorem~\ref{thmx:main}}
\label{subsec:proof}

We now recall and prove our main result.  In the following statement, we have permuted the parameters of the trisection for notational convenience in the proof.

\begin{reptheorem}
{thmx:main}
	Let $X$ be a 4--manifold that admits a $(g;k_1,0,k_3)$--trisection $\T$.  Then $\T$ is an irregular 3--fold cover of the standard trisection $\mathbb{T}_0$ of $S^4$. The branching set $\Ss$ is a closed, connected surface, which is smoothly embedded in $S^4$ away from one singular point, the cone on a link.
	
	If, in addition, $k_3=0$, the branching set $\Ss'$ can be chosen to be embedded in $S^4$. In this case, there are two singular points, each the cone on a knot.
\end{reptheorem}

\begin{proof}
	Let $\T$ be a $(g;k_1,0,k_3)$--trisection of $X$ given by the decomposition $X = Z_1\cup Z_2\cup Z_3$, with core $\Sigma$ and spine $H_1\cup H_2\cup H_3$.  Let $\T_0$ be the standard trisection $\T_0$ of $S^4$, with spine given by $B_1\cup B_2\cup B_3$, where each $B_i$ is a three-ball with boundary the core $S\cong S^2$. We will construct the desired branched covering map in stages. First, we apply Hilden's techniques to build a branched covering from the spine of $\T$ to the spine of $\T_0$. Next, we will modify this branched covering to control its branch locus.  Finally, we will extend the modified branched covering across the sectors of the trisections.
	
	We begin by identifying the cores of the two trisections, $\Sigma$ and $S$, with the standard genus--$g$ surface $\Sigma_g$ and the standard 2--sphere $S^2$, respectively.  Having done so, we view Hilden's covering map $\pi$ as having domain $\Sigma$, codomain $S$, and branch locus a fixed collection $\bold x$ of $2g+4$ points in~$S$.  We denote this set-up by $\pi\colon \Sigma\to(S,\bold x)$.
		
	By Theorem~\ref{thm:Hilden}, we can extend Hilden's map $\pi\colon \Sigma\to(S,\bold x)$ to an irregular 3--fold covering $\wh\pi_i\colon H_i\to(B_i,\Tt_i)$ for each $i\in\Z_3$.  This completes the first step of the proof outlined above.  The potential obstacle to extending this map over the sectors of $\T_0$ is that the unions
	$$(S^3_i,L_i) = (B_i,\Tt_i)\cup_{(S,\bold x)}\overline{(B_{i+1},\Tt_{i+1})}$$
	are uncontrolled. Indeed, a colored link in $S^3$ does not always bound a smooth colored surface in~$B^4$~\cite{kjorr2017admissible}. We will avoid this problem by arranging that $L_1$ and $L_3$ be unlinks.
	
	First, consider the covering
	$$\pi_1 = \wh\pi_1\cup_\pi\wh\pi_2\colon H_1\cup_\Sigma\overline{H_2}\to(B_1,\Tt_1)\cup_{(S,\bold x)}\overline{(B_2,\Tt_2)},$$
	and note that $(S^3_1,L_1) = (B_1,\Tt_1)\cup_{(S,\bold x)}\overline{(B_2,\Tt_2)}$ is a normalized bridge splitting of $L_1$ and $Y_1 = H_1\cup_\Sigma \overline{H_2}$ is a genus $g$ Heegaard splitting of $\#^{k_1}(S^1\times S^2)$.
	By Lemma~\ref{lemma:unlink-lift}, there is an irregular 3--fold covering
	$$\pi_1'\colon Y_1'\to(S^3_1,L_1'),$$
	where $L_1'$ is an unlink with $k_1+2$ components, equipped with a normalized $(g+2)$--bridge splitting
	$$(S^3_1,L_1') = (B_1',\Tt_1')\cup_{(S',\bold x')}\overline{(B_2',\Tt_2')},$$
	which lifts under the covering map to a genus $g$ Heegaard splitting $Y'_1 = H_1'\cup_{\Sigma'}\overline{H_2'}$, with $Y_1'\cong Y_1 \cong \#^{k_1}(S^1\times S^2)$. 
	
	By Waldhausen's Theorem~\cite{Wal_68_Heegaard-Zerlegungen-der-3-Sphare}, the two genus $g$ Heegaard splittings $H_1\cup_\Sigma\overline{H_2}$ and $H_1'\cup_{\Sigma'}\overline{H_2'}$ are homeomorphic.  Therefore, we can apply Theorem~\ref{thm:Pier_splitting} to conclude that $\Tt_2$ can be transformed via Piergallini moves and isotopy rel-$\partial$ to a tangle $(B_2,\Tt_2'')$ such that $L'' = \Tt_1\cup_\bold x\overline{\Tt_2''}$ is an unlink with $k_1+2$ components.  Let $\wh\pi_2''\colon H_2\to(B_2,\Tt_2'')$ denote the covering map corresponding to $\Tt_2''$.  Since $\wh\pi_2''$ and $\wh\pi_2$ agree when restricted to $\partial H_2$, we can swap out $\wh\pi_2$ for $\wh\pi_2''$.  Let $\pi_1'' = \wh\pi_1\cup\wh\pi_2''$.
	
	Importantly, we were able to change $\Tt_2$ and $\wh\pi_2$ above without altering $\wh\pi_1$ or $\Tt_1$.  Thus, we can repeat the above process to transform $\Tt_3$ into a tangle $\Tt_3''$ such that $\Tt_1\cup_\bold x\overline{\Tt_3''}$ is an unlink with $k_3+2$ components, swapping out $\wh\pi_3$ for a corresponding covering map $\wh\pi_3''$.  In this way, we obtain extensions $\wh\pi_1$, $\wh\pi_2''$, and $\wh\pi_3''$ of Hilden's map $\pi\colon\Sigma\to S$ across the respective handlebodies $H_1$, $H_2$, and $H_3$ comprising the spine of $\T$ whose branch loci $\Tt_1$, $\Tt_2''$, and $\Tt_3''$ have the property that $L_1 = \Tt_1\cup_\bold x\overline{\Tt_2''}$ and $L_3 = \overline{\Tt_1}\cup_\bold x\Tt_3''$ are both unlinks.  Let $\pi_2'' = \wh\pi_2''\cup\wh\pi_3''$, and let $\pi_3'' = \wh\pi_3''\cup\wh\pi_1$.

	To complete the proof, we must extend the covering $\pi_1''\cup\pi_2''\cup\pi_3''$ from the spine of $\T$ across the three sectors $Z_i$.  Since $L_1$ and $L_3$ are unlinks in normalized bridge position, they have a single blue component and at least one red component.  (Note $|L_i| = k_i+2$.)  This coloring can be extended  across trivial disk-tangles $\Dd_1$ and $\Dd_3$ bounded by these unlinks.  This determines extensions of $\pi_1''\cup\pi_2''\cup\pi_3''$ across $Z_1$ and $Z_3$.  (Technically, this gives extensions across 4--dimensional 1--handlebodies $Z_1'$ and $Z_3'$ with respective boundaries $Y_1$ and $Y_3$; however, the $Z_i'$ are diffeomorphic rel-$\partial$ to the $Z_i$, by Laudenbach-Poenaru~\cite{LauPoe_72_A-note-on-4-dimensional}, so we can regard these extensions as being across $Z_1$ and $Z_3$.)
	
	Finally, we note that $L_2'' = \Tt_2''\cup_\bold x\overline{\Tt_3''}$ has not been controlled and, at first glance, appears to be an arbitrary link.  However, we know that $L_2''$ is in normalized bridge position, and we know that $Y_2 = H_2\cup_\Sigma\overline{H_3}$ is diffeomorphic to $S^3$, since $k_2=0$.
	It follows that $\pi_2''$ gives an irregular 3--fold cover of $S^3$ to itself with branch locus $L_2''$.  By taking the cone of this covering map, we obtain a singular branched covering of $B^4$ to itself with branch locus the cone $\Dd_2$ on $L_2''$. Set $\Ss = \Dd_1\cup\Dd_2\cup\Dd_3$. This is the branching set of our cover and and we see that $\Ss$ is a smoothly embedded surface  in $S^4$ away from a single singularity, which is the cone on the link~$L_2''$.  This completes the proof of the first statement.

	We now address the case when $k_3=0$.  In this case, we proceed as before, arranging that $L_1''$ be an unlink, and noting that this involves changing only $\Tt_2$ and $\wh\pi_2$.  Next, we apply Theorem~\ref{thm:Hilden} to $\wh\pi_3$ and $\Tt_3$ to control the pairing of the bridge points $\bold x$ induced by this tangle.
	This gives a new covering $\wh\pi_3'\colon H_3\to (B_3,\Tt_3'')$ such that both $L_2'' = \Tt_2''\cup\overline{\Tt_3''}$ and $L_3'' = \Tt_3''\cup_\bold x\overline{\Tt_1}$ are connected -- i.e. knots.
	(It is a simple exercise to verify that there is a pairing of the points of $\bold x$ that assures this. In fact, this follows from the fact that $(b;c_1,1,1)$--bridge trisections exist for any choice of $c_1$ and $b>c_1$; see Figure~\ref{fig:pairingunlinkbridge} for an example of such a pairing.)
	Then, we choose $\Dd_1'$ to be trivial disks for $L_1''$, and we choose $\Dd_2'$ and $\Dd_3'$ to be cones on $L_2''$ and $L_3''$, respectively. This allows us to extend $\pi$ across the entire trisection $\T$, as desired.  Let $\Ss' = \Dd_1'\cup\Dd_2'\cup\Dd_3'$; this is the desired branching set containing two singularities that are each the cone on a knot.
\end{proof}

\begin{figure}[h!]
\includegraphics[scale=.5]{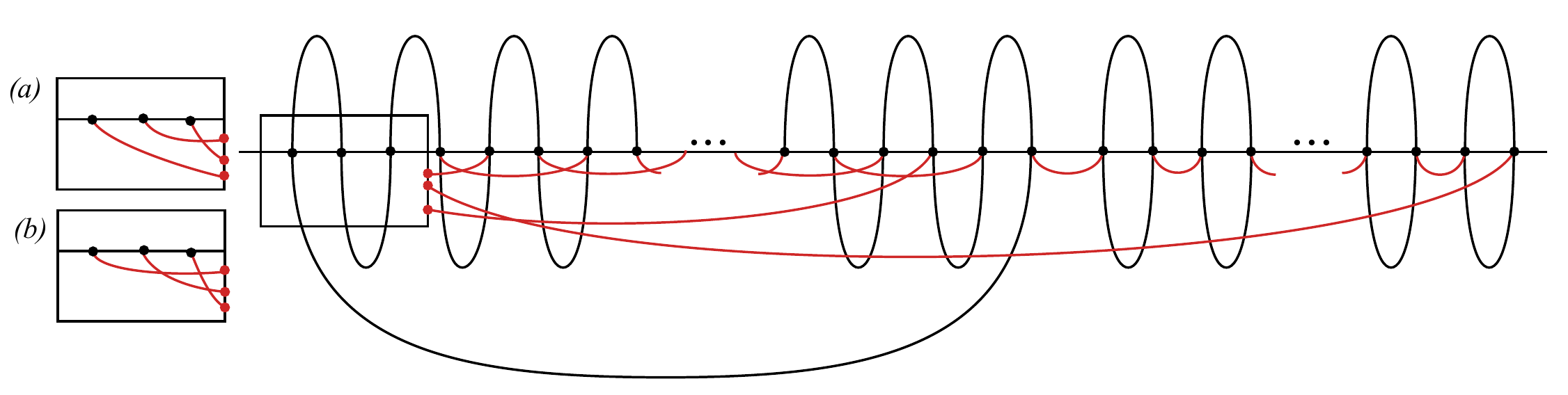}
\caption{A $b$--bridge splitting of an unlink, together with red arcs illustrating a pairing of the boundary points whose union with either tangle is connected. The boxed portion of the figure is given by (a) or (b) based on whether the number of maxima in the leftmost component of the unlink is odd or even, respectively.}
\label{fig:pairingunlinkbridge}
\end{figure}

To justify Remark~\ref{rmk:euler}, we note that the bridge trisection of $\Ss$ constructed in the proof above induces a cellular decomposition of $\Ss$ that contains $2(g+2)$ zero-cells, $3(g+2)$ one-cells, and $k_1+k_2+4$ two-cells, plus the (contractible) cone $\Dd_2$ on $L_2''$.  It follows that $\chi(\Ss) = k_1+k_2-g+3$, as claimed. 
Similarly, the bridge trisection of $\Ss'$ induces a cellular decomposition that contains $2(g+2)$ zero-cells, $3(g+2)$ one-cells, and $k_1+2$ two-cells, plus the two cones $\Dd_2'$ and $\Dd_3'$ on the knots $L_2''$ and $L_3''$.  It follows that $\chi(\Ss') = k_1-g+2$, as claimed.
The connection to the Euler characteristic of $X$ follows from the following fact:  If $X$ admits a $(g;k_1,k_2,k_3)$--trisection, then
$$\chi(X) = 2+g-k_1-k_2-k_3.$$

\begin{remark}
\label{rmk:smooth}
	In the above proof, if $L_3''$ can be arranged to be an unlink, then it has two components, since it is in normalized bridge position and its irregular 3--fold cover is $S^3$.  In this case, the cone $\Dd_3$ on $L_3''$ can be swapped out for a 2--component trivial disk-tangle.  After this modification, $\Ss$ becomes a smoothly embedded surface of Euler characteristic $k_1+k_2-g+4$.
	
	Changing the number of singularities in the construction of $f$ can affect the orientability of the branching set $\Ss$, as seen from Euler characteristic and signature considerations. When $\Ss$ has a trivial normal bundle, the number of singular points is congruent mod~2 to the signature $\sigma(X)$ of the covering space, since the contribution of each singularity to  $\sigma(X)$ is an odd integer~\cite{cahnkju2018genus,kjuchukova2018dihedral}.	
\end{remark}

\begin{question}
\label{ques:improve_sing}
	Can the proof of Theorem~\ref{thmx:main} be adapted to ensure that the singularity on the branching set is the cone on a knot? Can the singularity be removed altogether?
\end{question}

\section{Examples}
\label{examples}

In this section, we give some examples of irregular 3--fold covers of $S^4$. We use Fox 3--colored tri-plane diagrams to depict the branching sets and associated dihedral representations. This is sufficient to determine trisection diagrams of the corresponding irregular 3--fold covers~\cite{cahnkjuchukova2017singbranchedcovers}.

\begin{figure}[h!]
	\includegraphics[width=\textwidth]{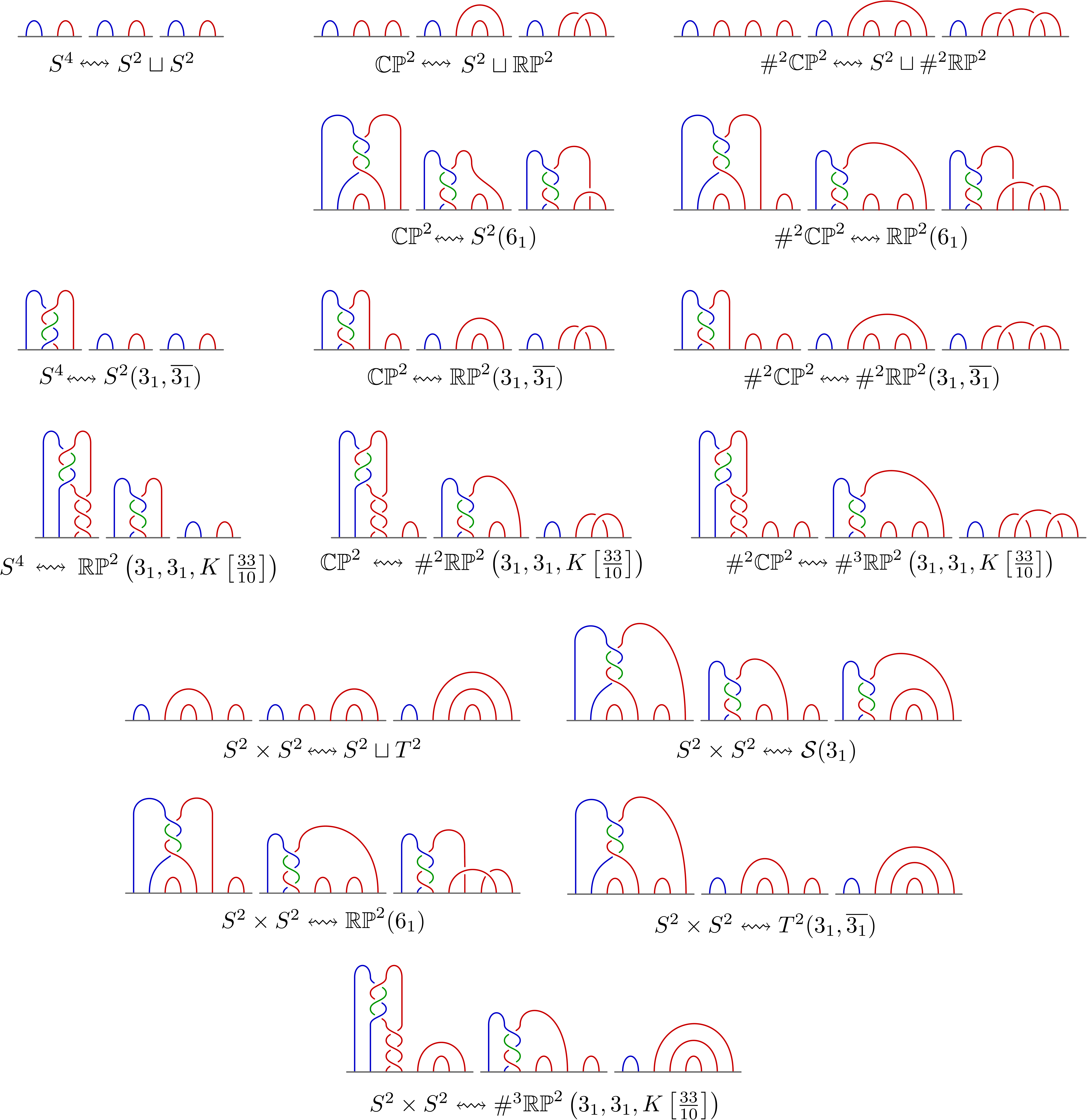}
	\caption{Tricolored tri-plane diagrams of some irregular 3--fold covers of $S^4$. Below each diagram is the total space of the corresponding cover; the homeomorphism type of the branching set; the singularities of the embedding.  }
	\label{fig:triplanes}
\end{figure}

\begin{example}
\label{ex:basic}
	Figure~\ref{fig:triplanes} gives Fox 3--colored tri-plane diagrams of the branching sets for basic examples of irregular 3--fold coverings of $S^4$.  The notation ``$X\leftrightsquigarrow\Ss$'' means that $X$ is the irregular 3--fold cover of $S^4$, branched along a singular surface homeomorphic to $\Ss$ and embedded in $S^4$ according to the given tri-plane diagram.  The singularities of the given embedding are described in the parentheses; for example, $\RP^2\left(3_1, 3_1, K\left[\frac{33}{10}\right]\right)$ represents a projective plane with three singularities, two of which are cones on right-handed trefoils and one of which is a cone on the 2--bridge knot $K\left[\frac{33}{10}\right]$.  The surface $\Ss(3_1)$ is the spun trefoil, which is a smoothly embedded, knotted 2--sphere.
	
	In each of the top four rows, moving left to right corresponds to taking the connected sum of the previous branch locus with a monochromatic projective plane.  This changes the cover by taking the connected sum with $\CP^2$.  This summand can be chosen to have either orientation by a change of crossing in the third tangle.  Passing from one figure to the the one below it corresponds to increasing the number of singularities of the branch locus and lowering the Euler characteristic by one.  In each case, this is accomplished via a 3--move~\cite{montesinos1985note}.	
	The branch loci in the first and fifth rows are smoothly embedded.  All singularities are cones on knots.
\end{example}

\begin{example}
[\cite{cahnkju2018genus}]
	Given an integer $n\geq 0$, there exists a 3--fold cover $$f_{2n+1}\colon \#^{2n+1}\mathbb{CP}^2\to S^4$$ with branching set an embedded orientable surface $\Ss_n\subset S^4$ of genus $n$ and one two-bridge singularity. The map $f_{2n+1}$ induces the standard  genus $2n+1$ trisection of $\#^{2n+1}\mathbb{CP}^2$.  
\end{example}

In~\cite{LamMei_18_Bridge-trisections}, $(g,0)$--trisections are constructed for many complex algebraic surfaces, including K3.  Since these trisections satisfy the hypotheses of Theorem~\ref{thmx:main}, it is natural to desire an explicit construction of the resulting cover.

\begin{problem}
\label{ques:algebraic}
	Find explicit descriptions of the bridge trisected surfaces whose irregular 3--fold covers are these $(g,0)$--trisections of algebraic surfaces.
\end{problem}

\bibliographystyle{amsalpha}
\bibliography{covers}

\end{document}